\documentclass{amsart}

\usepackage{amssymb}
\usepackage{amsthm}
\usepackage{amsmath}
\usepackage{amsbsy}
\usepackage[all]{xy}
\usepackage{bm}
\usepackage{hyperref}
\usepackage{tikz}
\usepackage{array}
\usepackage{colortbl,xcolor}
\usepackage{ytableau}
\usepackage{hhline}

\allowdisplaybreaks

\usepackage{thmtools}
\usepackage{thm-restate}
\usepackage{mathrsfs}

\usepackage{caption}
\newtheorem{theorem}{Theorem}[section]
\newtheorem{proposition}[theorem]{Proposition}
\newtheorem{corollary}[theorem]{Corollary}

\newtheorem{conjecture}[theorem]{Conjecture}
\newtheorem{lemma}[theorem]{Lemma}

\newtheorem{question}[theorem]{Question}
\newtheorem{problem}[theorem]{Problem}

\DeclareMathOperator{\ext}{\operatorname{le}}
\DeclareMathOperator{\LE}{\mathbf{LE}}

\begin{document}

\title[]{Linear extension numbers of $n$-element posets} 

\author[]{Noah Kravitz}
\address[]{Grace Hopper College, Yale University, New Haven, CT 06510, USA}
\email{noah.kravitz@yale.edu}
\author[]{Ashwin Sah}
\address[]{Massachusetts Institute of Technology, Cambridge, MA 02139 USA}
\email{asah@mit.edu}

\begin{abstract} 
We address the following natural but hitherto unstudied question: what are the possible linear extension numbers of an $n$-element poset?  Let $\LE(n)$ denote the set of all positive integers that arise as the number of linear extensions of some $n$-element poset.  We show that $\LE(n)$ skews towards the ``small'' end of the interval $[1,n!]$.  More specifically, $\LE(n)$ contains all of the positive integers up to $\exp\left(c\frac{n}{\log n}\right)$ for some absolute constant $c$, and $|\LE(n) \cap ((n-1)!,n!]|<(n-3)!$.  The proof of the former statement involves some intermediate number-theoretic results about the Stern-Brocot tree that are of independent interest.
\end{abstract}
\maketitle

\section{Introduction}\label{sec:introduction}

\subsection{Background and main question}

A \textit{partially ordered set} (or \textit{poset}) consists of a ground set $P$ together with a transitive, antisymmetric, and reflexive binary relation $\leq_P$ on $P$.  Posets arise naturally in many areas of math and have been studied extensively.  In this paper, we concern ourselves with posets where the base set has a finite number of elements.
\\

A \textit{linear extension} of the poset $(P,\leq_P)$ is a total ordering $\leq^{\ast}$ on $P$ which ``extends'' the relation $\leq_P$ in the sense that $x \leq_P y$ necessarily implies $x \leq^{\ast} y$.  Throughout this paper, we let $\ext(P)$ denote the number of linear extensions of the poset $P$ (with $\leq_P$ implicit).  The problem of computing $\ext(P)$ for particular posets $P$ has received significant attention in a variety of contexts; see, e.g., \cite{atkinson1985extensions,bjorner1991permutation,brightwell2003number,felsner2015linear}.  Linear extension numbers have also been studied in relation to comparability graphs, convex polytopes, and other combinatorial objects in \cite{edelman1989recurrence,stachowiak1989relation,stanley1986two}.
\\

The following straightforward recursive algorithm generates all linear extensions of a poset $P$: for each minimal element $x \in P$ (i.e., an element $x$ such that there is no $y<_P x$), take $x \leq^{\ast} z$ for all $z \in P$ and then find all linear extensions of the poset induced by $\leq_P$ on the set $P \setminus \{x\}$.  One could also proceed by choosing any pair of incomparable elements $(x,y)$ and (recursively) finding all linear extensions of the two posets obtained by adding either the relation $x \leq_P y$ or the relation $y \leq_P x$.
\\

The computation of linear extension numbers is quite difficult in general.  For instance, Brightwell and Winkler \cite{brightwell1991counting} showed that the problem of computing $\ext(P)$ for arbitrary $P$ is NP-complete.  See \cite{dittmer2018counting} for a more recent perspective.
\\

Until now, there has not been a comparable interest in the natural inverse question for linear extension numbers: for fixed $n$, which positive integers can appear as $\ext(P)$, where $|P|=n$?  In this paper, we investigate the properties of the set
\[\LE(n)=\{\ext(P): |P|=n\}\]
for various values of $n$.
\\

Trivially, $1 \leq \ext(P)\leq n!$.  As $P$ ranges over all $n$-element posets, one would expect $\ext(P)$ to take on small values more readily than large ones.  For instance, it is easy to find posets with $\ext(P)=k$ for all $1 \leq k \leq n$ (by taking $P$ to have a chain of length $n-1$), whereas we can never have $\frac{n!}{2}<\ext(P)<n!$ (since the existence of even a single nontrivial relation forces $\ext(P)\leq \frac{n!}{2}$).  Note also that $\LE(n)\subseteq\LE(n+1)$ since we can simply add an element that is smaller than all other elements.

\subsection{Main results}

Heuristically speaking, we strengthen the above observations into a nontrivial description of how $\LE(n)$ skews towards small elements of the interval $[1, n!]$. Our main result is that $n$-element posets achieve all small linear extension numbers in the following sense.

\begin{theorem}[Small linear extension numbers]\label{thm:small}
There is an absolute constant $c$ such that for all sufficiently large $n$, the set $\LE(n)$ contains all positive integers up to $\exp\left(c\frac{n}{\log n}\right)$.
\end{theorem}

We prove this theorem by recursively constructing a family of width-$2$ posets whose linear extension numbers are related to the entries of the Stern-Brocot tree.  We then control the appearance of these entries by relating this phenomenon to certain properties of the Euclidean algorithm. We mention the following intermediate result here because it is of independent number-theoretic interest.

\begin{theorem}\label{thm:number-theoretic}
There is an absolute constant $c$ such that the following holds: for every $n\ge 2$, there exists some $1\le d\le n - 1$ relatively prime to $n$ such that when the Euclidean algorithm is run on the pair $(n, d)$, the sum of the quotients obtained is at most $c\frac{n}{\phi(n)}\log n\log\log n$, where $\phi$ denotes Euler's totient function.
\end{theorem}
We conjecture that the factor of $\frac{n}{\phi(n)}\log\log n$ can be removed.
\\

We complement Theorem~\ref{thm:small} by showing that $n$-element posets avoid most large linear extension numbers in the following sense.

\begin{theorem}[Large linear extension numbers]\label{thm:large}
For all $n\geq 8$, we have $|\LE(n) \cap ((n-1)!,n!]|<(n-3)!$.
\end{theorem}
We prove this theorem by analyzing the connected components of Hasse diagrams.  The key insight is that an $n$-element poset with a connected Hasse diagram has at most $(n-1)!$ linear extensions.
\\

\subsection{Structure of the paper}

In Section~\ref{sec:general-constructions}, we describe general constructions for manipulating the linear extension counts of posets.  In Section~\ref{sec:nice-family}, we use these ideas to construct a specific family of posets whose linear extension numbers are connected to the Stern-Brocot tree. In Section~\ref{sec:correspondence}, we use known facts about the Stern-Brocot tree to reduce Theorem~\ref{thm:small} to Theorem~\ref{thm:number-theoretic}, which we prove in Section~\ref{sec:proofs}.  In Section~\ref{sec:large}, we prove Theorem~\ref{thm:large} and related results.  Finally, we pose some conjectures and questions in Section~\ref{sec:conclusion}.

\subsection{Conventions and notation}\label{sec:conventions}
Throughout this paper, the set of natural numbers is $\mathbb{N} = \{1, 2, \ldots\}$, and $\mathbb{N}^\ast$ denotes the set of all (possibly empty) finite sequences of natural numbers. Logarithms are natural unless otherwise specified, and $\phi$ always denotes Euler's totient function.
\\

Two elements $x$ and $y$ in a poset $P$ are \textit{comparable} if either $x \leq_P y$ or $y \leq_P x$; otherwise, these elements are \textit{incomparable}.  A \textit{chain} in a poset $P$ is a subset in which any two elements are comparable, and an \textit{antichain} is a subset in which no two distinct elements are comparable.  We write $\textbf{1}$ for the \textit{singleton} poset whose base set consists of only a single element.  For $x, y \in P$, we write $x<_P y$ if $x \leq_P y$ and $x \neq y$.  We say that $y$ \textit{covers} $x$ if $y>_P x$ and there is no $z \in P$ such that $y>_P z>_P x$.
\\

The \textit{Hasse diagram} is a useful way to visualize a poset $P$.  The Hasse diagram of $P$ consists of points corresponding to the elements of $P$, where $y$ is placed higher than $x$ whenever $y>_P x$.  If $y$ covers $x$, then we also connect the points $x$ and $y$ with a line segment.

\section{Some general poset constructions}\label{sec:general-constructions}
There are several natural ways to build larger posets out of smaller posets such that the linear extension numbers of the larger posets can be understood in terms of the linear extension numbers of the smaller posets.
\\

Given posets $(P, \leq_P)$ and $(Q, \leq_Q)$, the \textit{disjoint sum} of $P$ and $Q$, written $P+Q$, is the poset on the base set $P \cup Q$, where $x\leq_{P+Q}y$ if and only if one of the following holds:
\begin{enumerate}
    \item $x,y \in P$ and $x\leq_P y$
    \item $x,y \in Q$ and $x\leq_Q y$.
\end{enumerate}
The Hasse diagram of $P+Q$ is obtained by placing the Hasse diagrams of $P$ and $Q$ side-by-side.  It is clear that $$\ext(P+Q)=\binom{|P|+|Q|}{|P|}\ext(P)\ext(Q).$$ This is the most basic way to combine two posets.
\\

The \textit{direct sum} of $(P, \leq_P)$ and $(Q, \leq_Q)$, written $P \oplus Q$, is the poset on the base set $P \cup Q$, where $x\leq_{P \oplus Q}y$ if and only if one of the following holds:
\begin{enumerate}
    \item $x,y \in P$ and $x\leq_P y$
    \item $x,y \in Q$ and $x\leq_Q y$
    \item $x \in P$ and $y \in Q$.
\end{enumerate}
The Hasse diagram of $P \oplus Q$ is obtained by placing the Hasse diagram of $Q$ above the Hasse diagram of $P$ and connecting each maximal element of $P$ to each minimal element of $Q$.  It is clear that $$\ext(P \oplus Q)=\ext(P)\ext(Q).$$  This property is useful for constructing small posets with large composite linear extension numbers.
\\

For any two incomparable elements $x$ and $y$ of a poset $P$, let $P[x<y]$ denote the poset that is obtained from $P$ by adding the relation $x<y$.  Then, as noted above,
$$\ext(P)=\ext(P[x<y])+\ext(P[y<x]).$$
In particular, suppose $M$ is a maximal element of the poset $P$ and $m$ is a minimal element of the poset $Q$.  Then we define the $(M,m)$\textit{-direct sum} of $P$ and $Q$, written $P \oplus_{M,m} Q$, to be the poset obtained from $P \oplus Q$ by deleting the relation $M<m$.  We can now compute the number of linear extensions of this new poset:
\begin{align*}
\ext(P \oplus_{M,m} Q)&=\ext((P \oplus_{M,m} Q)[M<m])+\ext((P \oplus_{M,m} Q)[m<M])\\
 &=\ext(P \oplus Q)+\ext((P\setminus \{M\})\oplus \{m\} \oplus \{M\} \oplus (Q \setminus \{m\}))\\
  &=\ext(P)\ext(Q)+\ext(P\setminus \{M\})+\ext(Q\setminus \{m\}),
\end{align*}
where $P\setminus \{x\}$ is the restriction of the poset $(P,\leq_P)$ to the set $P \setminus \{x\}$.

\section{A nice family of posets}\label{sec:nice-family}
We now use the $(M,m)$-direct sum to construct a large family of width-$2$ posets whose linear extension numbers we can control.  We will write this family as $\{P_b\}$, where $b$ ranges over the dyadic rationals in the interval $[0,1)$.  (In this section, all decimals are in base $2$.)  We construct this family recursively based on the number of digits needed to express $b$.  If $b$ requires $n$ digits (including the digit before the decimal point), then $|P_b|=n$.

\subsection{The construction}
For $n\geq 2$, each poset $P_b$ in the $n$-th stage gives rise to the posets $P_{b-2^{-n}}$ and $P_{b+2^{-n}}$ in the $n+1$-th stage.  Each of these posets is obtained from $P_b$ as the $(M,m)$-direct sum of the singleton poset $\textbf{1}$ with $P_b$ .  So each new poset looks like a copy of $P_b$ with a new minimal element added to one of the chains.
\\

For each $P_b$, we keep track of which elements are in which chain.  We will let $L_b$ and $R_b$ denote the minimal elements of the left and right chains, respectively.  We present each step of the construction ``in symbols'' and then ``in words.''
\\

The $n=1$ stage consists of just the poset $P_0$, which we define to be the singleton
$$P_0=\{R_0\}.$$
Note that there is no $L_0$ since $P_0$ contains only one element.
\\

The $n=2$ stage consists of just the poset $P_{0.1}$, which we define to be
$$P_{0.1}=\{L_{0.1}, R_{0.1}\}, \quad \text{with} \quad L_{0.1}<R_{0.1}.$$
Note that even though $P_{0.1}$ is itself a chain, we have identified the two non-maximal chains $\{L_{0.1}\}$ and $\{R_{0.1}\}$.
\\

We now explain the recursive part of the construction.  For $n \geq 2$, fix any $P_b$ in the $n$-th stage of the construction.  Define
$$P_{b-2^{-n}}=\{x\} \oplus_{x, L_b} P_b,$$
and re-label $x$ as $R_{b-2^{-n}}$ and $L_b$ as $L_{b-2^{-n}}$.  In words: we add a new minimal element to $P_b$ below $R_b$ which is smaller than everything except $L_b$.
\\

In order to construct $P_{b+2^{-n}}$, we consider two cases for $b$.  If $b$ is not of the form $2^{1-n}$, then (just as before) define
$$P_{b+2^{-n}}=\{x\} \oplus_{x, R_b} P_b,$$
and re-label $x$ as $L_{b+2^{-n}}$ and $R_b$ as $R_{b+2^{-n}}$.  In words: we add a new minimal element to $P_b$ below $L_b$ which is smaller than everything except $R_b$.  If $b=2^{1-n}$, then this construction does not work because $R_b$ is not minimal and the $(x, R_b)$-direct sum is not defined.  In this case, we define the analogous construction manually: let
$$P_{b+2^{-n}}=\{x\} \oplus P_b,$$
and re-label $x$ as $L_{b+2^{-n}}$ and $R_b$ as $R_{b+2^{-n}}$.  In words: we try to add a new minimal element to $P_b$ below $L_b$ which is smaller than everything except $R_b$, but this new element is automatically smaller than $R_b$ because $R_b$ was not minimal.  Figure~\ref{fig:big} shows the first few stages of this construction.
\\

\begin{figure}[p]
\begin{center} 
\includegraphics[height=17.5cm]{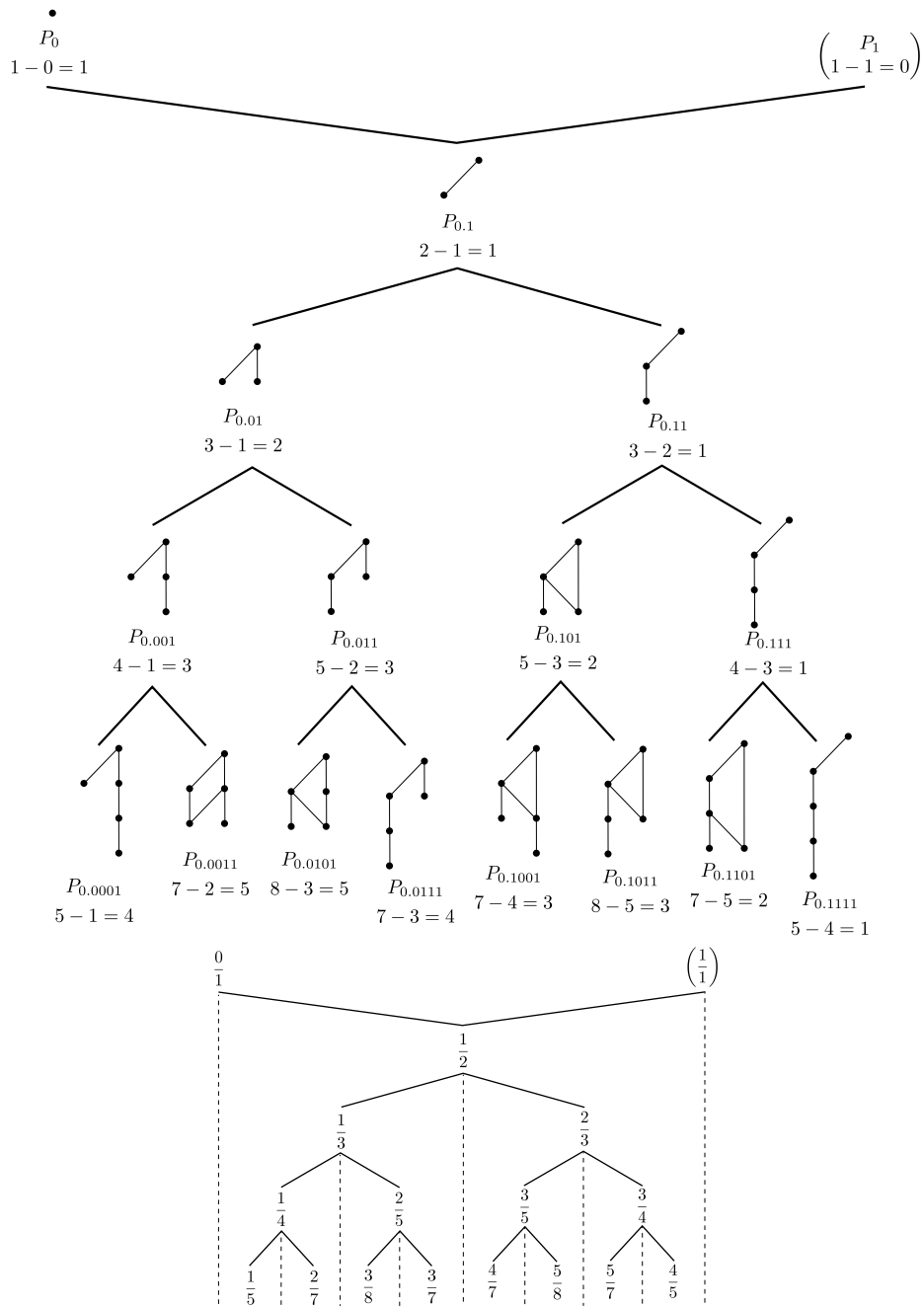}
\end{center}
\caption{The first $5$ stages of the construction of the posets $\{P_b\}$ (above) and the corresponding section of the Stern-Brocot tree (below).  Each $\ext(P_b)$ is written as the difference of the denominator and numerator of the corresponding entry in the Stern-Brocot tree.  We indicate where $P_1$ would fit if it were defined.  In order to emphasize the width-$2$ nature of these posets, we include more edges than are, strictly speaking, required for the Hasse diagrams.}
\label{fig:big}
\end{figure}

We remark that $P_{0.1}$ can be obtained from $P_0$ as
$$P_{0.1}=\{x\} \oplus P_0,$$
after re-labeling $x$ as $L_{0.1}$ and $R_0$ as $R_{0.1}$.  We included $P_{0.1}$ in the base case because $P_0$ does not give rise to a second poset in the $n=2$ stage.

\subsection{Tree structure}
The following lemma relates the linear extension numbers of posets in later stages to the linear extension numbers of posets in earlier stages.

\begin{proposition}\label{prop:recursion}
Let $b$ be a dyadic rational in the range $(0,1)$.  If $P_b$ first appears in the $n$-th stage, i.e., $b=0.b_2b_3\cdots b_{n-1}1$, then
$$\ext(P_b)=\ext(P_{b-2^{1-n}})+\ext(P_{b+2^{1-n}}),$$
where we convene to let $\ext(P_1)=0$.
\end{proposition}

It is easy to convince oneself of the truth of this statement by scrutinizing Figure~\ref{fig:big}: for a given poset $P_b$, consider the two cases that arise from adding a relation between $L_b$ and $R_b$, then follow the recurrence ``up'' the tree by eliminating new minimal elements one-by-one.  For the sake of completeness, we provide a technical, symbol-heavy proof as well.

\begin{proof}
Consider the following four cases for $b$:
\begin{enumerate}
    \item $b=0.11\cdots 1=1-2^{1-n}$.
     \item $b=0.b_2 \cdots b_{n-k-1}011 \cdots 1$ (where $b$ ends with $k\geq 2$ $1$'s, preceded by a $0$).
     \item $b=0.00\cdots 01=2^{1-n}$.
    \item $b=0.b_2 \cdots b_{n-k-2}100\cdots 01$ (where $b$ ends with a $1$, preceded by exactly $1 \leq k \leq n-3$ $0$'s).
\end{enumerate}

First, consider $b=0.1\cdots 1=1-2^{1-n}$.  By easy induction on $n$, we have
$$P_{1-2^{1-n}}=\{a_1\} \oplus \cdots \oplus \{a_n\},$$
which immediately implies that
$$\ext(P_{1-2^{1-n}})=1$$
for all $n \geq 1$.  Finally, for all $n \geq 2$, we have
$$\ext(P_{b-2^{1-n}})+\ext(P_{b+2^{1-n}})=\ext(P_{1-2^{2-n}})+\ext(P_1)=1+0=\ext(P_b),$$
as desired.
\\

Second, consider $b=0.b_2 \cdots b_{n-k-1}011 \cdots 1$ (where $b$ ends with $k\geq 2$ $1$'s, preceded by a $0$).  We proceed by induction on $k$.  We begin with the base case $k=2$, i.e., $b=0.b_2 \cdots b_{n-3}011$.  Let $b''=b+2^{1-n}=0.b_2 \cdots b_{n-3}1$.  In the recursive construction, $P_{b''}$ gives rise to $P_{b'}$, where $b'=b''-2^{2-n}=0.b_2 \cdots b_{n-3}01$, via
$$P_{b'}=\{R_{b'}\} \oplus_{R_{b'}, L_{b''}} P_{b''}.$$
In turn, $P_{b'}$ gives rise to $P_b$, where $b=b'+2^{1-n}$, via
$$P_{b}=\{L_{b}\} \oplus_{L_{b}, R_{b'}} P_{b'}.$$
We now compute
\begin{align*}
    \ext(P_b) &=\ext(\{L_{b}\})\ext(P_{b'})+\ext(\emptyset)\ext(P_{b'} \setminus \{R_{b'}\})\\
     &=\ext(P_{b'})+\ext(P_{b''})\\
     &=\ext(P_{b-2^{1-n}})+\ext(P_{b+2^{1-n}}),
\end{align*}
as desired.  Next, we perform the induction step.  Consider $b=0.b_2 \cdots b_{n-k-1}011 \cdots 1$, where now $k \geq 3$.  Let $b'=b-2^{1-n}$ (deleting the last $1$ in the decimal expansion of $b$) and $b''=b'-2^{2-n}$ (deleting the penultimate $1$, as well).  As before,
$$P_{b'}=\{L_{b'}\} \oplus_{L_{b'}, R_{b''}} P_{b''}$$
gives
$$\ext(P_{b'})=\ext(P_{b''})+\ext(P_{b''}\setminus \{R_{b''}\}).$$
By the induction hypothesis, we also have
$$\ext(P_{b'})=\ext(P_{b''})+\ext(P_{b'+2^{2-n}}),$$
so that
\begin{equation}\label{eq:tree}
    \ext(P_{b''}\setminus \{R_{b''}\})=\ext(P_{b'+2^{2-n}}).
\end{equation}
Next, consider
$$P_{b}=\{L_{b}\} \oplus_{L_{b}, R_{b'}} P_{b'},$$
which gives
$$\ext(P_{b})=\ext(P_{b'})+\ext(P_{b'}\setminus \{R_{b'}\}).$$
Recall that in $P_{b'}$, the element $L_{b'}$ is smaller than everything except for $R_{b'}$ (which is identified with $R_{b''}$ in $P_{b''}$).  Thus, we see that
$$P_{b'}\setminus \{R_{b'}\}=\{L_{b'}\} \oplus (P_{b''}\setminus \{R_{b''}\}),$$
which of course means that
$$\ext(P_{b'}\setminus \{R_{b'}\})=\ext(P_{b''}\setminus \{R_{b''}\}).$$
Finally, substituting from \eqref{eq:tree} gives
$$\ext(P_{b})=\ext(P_{b'})+\ext(P_{b'+2^{2-n}}).$$
This is the desired expression once one recalls that $b'=b-2^{1-n}$ and $b'+2^{2-n}=b+2^{1-n}$.
\\

The remaining two cases are completely analogous to the second case.
\end{proof}

\section{From posets to the Euclidean algorithm and back}\label{sec:correspondence}

At this point, the reader may recognize the structure of the (left half of the) Stern-Brocot tree in Figure~\ref{fig:big}.  We will exploit this connection in order to apply number-theoretic results about the Stern-Brocot tree to our problem about linear extension numbers.

\subsection{The appearance of denominators in the Stern-Brocot tree}
The Stern-Brocot tree is a complete binary plane tree that contains each positive rational exactly once: the rationals less than $1$ appear on the left half, and the rationals greater than $1$ appear on the right half.  (The entry $\frac{1}{0}$, representing infinity, also appears in the right half.)  The Stern-Brocot tree is closely related to the Calkin-Wilf tree, Stern's diatomic array, and Farey sequences: for instance, the Calkin-Wilf tree is obtained by permuting the entries in each layer of the Stern-Brocot tree \cite{calkin2000recounting}, and the denominators of the first $\ell$ layers of the left half of the Stern-Brocot tree are exactly the elements of the $\ell$-th line of Stern's diatomic array \cite{lehmer1929stern}.
\\

The entries of the left half of the Stern-Brocot tree can be parameterized in the obvious way by $\left\{ \frac{s_b}{t_b} \right\}$, where $b$ ranges over all dyadic rationals in $[0,1]$.  More precisely, the entries in the $\ell$-th layer are parametrized from left to right by the dyadic rationals in $[0,1]$ that require exactly $\ell$ digits in their decimal expansions.  This parameterization can be extended in the natural way to the entire Stern-Brocot tree by letting $b$ range over the dyadic rationals in $[0,2]$.
\\

It is well known \cite{aiylam2017generalized} that each layer of the Stern-Brocot tree is obtained by taking mediants of all consecutive terms of the entries in earlier layers; moreover, these mediants are already reduced fractions.  In other words, both the numerators and denominators in the Stern-Brocot tree individually follow the recursive structure of Proposition~\ref{prop:recursion}.  Since $\ext(P_0)=t_0-s_0$ and $\ext(P_1)=t_1-s_1$ (as defined), it then follows that $\ext(P_b)=t_b-s_b$ for all dyadic rationals $b$ in $[0,1]$.
\\

By the definition of the Calkin-Wilf tree (see \cite{calkin2000recounting}), each entry $\frac{s}{t}$ in the $\ell$-th layer has $\frac{s}{s+t}$ as its left child in the $\ell+1$-th layer.  Recall that the Calkin-Wilf and Stern-Brocot trees differ only in the ordering of elements within each layer.  These facts together imply that for every denominator $d$ in the $\ell$-th layer of the Stern-Brocot tree, there is some $\frac{s_b}{t_b}$ (with $s_b<t_b$, i.e., in the left half) in the $\ell+1$-th layer such that $t_b-s_b=d$.  In particular, every denominator $d$ in the $\ell-1$-th layer of the Stern-Brocot tree appears as the number of linear extensions of some $P_b$ on $\ell$ elements.
\\

It is also known \cite{lehmer1929stern} that any relatively prime positive integers $n,d$ appear consecutively somewhere in Stern's diatomic array.  In particular, the entries $n,d$ appear consecutively in the $\ell-1$-th line of Stern's diatomic array, where $\ell$ equals the sum of the quotients that appear in the expansion of $\frac{n}{d}$ as a continued fraction.  Recall that the sum of the quotients in the continued fraction expansion of $\frac{n}{d}$ is simply the sum of the quotients obtained in running the Euclidean algorithm on the pair $(n,d)$.  We summarize all of these observations in the following proposition.

\begin{proposition}\label{prop:connection}
Let $d<n$ be relatively prime positive integers, and let $\ell$ denote the sum of the quotients obtained when the Euclidean algorithm is run on the pair $(n,d)$.  Then there is some $P_b$ on at most $\ell$ elements with exactly $n$ linear extensions.
\end{proposition}

\subsection{Connecting Theorems \ref{thm:small} and \ref{thm:number-theoretic}}

We now use the ideas of the previous subsection to show how Theorem~\ref{thm:number-theoretic} (on sums of quotients in the Euclidean algorithm) implies Theorem~\ref{thm:small} (on small linear extension numbers).

\begin{proof}[Proof (Theorem~\ref{thm:number-theoretic} implies Theorem~\ref{thm:small})]
Recall that Theorem~\ref{thm:number-theoretic} asserts the existence of an absolute constant $c$ such that for every $n \geq 2$, there exists some $1 \leq d \leq n$ such that the sum of the quotients obtained in running the Euclidean algorithm on $(n,d)$ is at most $c \frac{n}{\phi(n)}\log n \log \log n$.  When $n=p$ is prime, we have $\frac{p}{\phi(p)}=\frac{p}{p-1}\leq 2$, so in this case the sum of quotients is at most $2c \log p \log \log p$.  Next, Proposition~\ref{prop:connection} tells us that for every prime $p$, there exists a poset $P_b$ of size at most $2c \log p \log \log p$ that has exactly $p$ linear extensions.
\\

Let $m \geq 2$ have prime factorization $m=p_1\ldots p_r$ (written with multiplicity).  For each $p_j$, let $P_{b_j}$ be the poset with exactly $p_j$ linear extensions from the previous paragraph.  Then the direct sum of the $p_j$'s has exactly $m$ linear extensions and size at most
$$\sum_{j=1}^r 2c \log p_j \log \log p_j \leq 2c\log \log m \sum_{j=1}^r \log p_j=2c\log \log m \log m.$$

Now, fix any positive integer $n$.  Every positive integer $m$ with $$2c\log m \log \log m\leq n$$ is obtained as the number of linear extensions of some poset on at most $n$ elements.  Moreover, because the sets $\{\LE(n)\}$ form an ascending chain (by taking the direct sum with a singleton), there is in fact an $n$-element poset that achieves each of these linear extension numbers $m$.  Finally, for $m \leq \exp(\frac{1}{4c} \frac{n}{\log n})$, we compute:
\begin{align*}
    2c\log m \log \log m &\leq 2c \left( \frac{1}{4c} \frac{n}{\log n} \right)\left( \log n-\log \log n-\log(4c) \right)\\
     &=\frac{n}{2}+O\left( \frac{n \log \log n}{\log n} \right).
\end{align*}
This quantity is smaller than $n$ for sufficiently large $n$, which suffices to establish Theorem~\ref{thm:small}.
\end{proof}

\section{Bounding the sum of quotients}\label{sec:proofs}
\subsection{Basic number-theoretic notions}\label{sec:basic-number-theory}
Given positive integers $n$ and $d$, the \emph{Euclidean algorithm} is defined as follows: let $a_0 = n$ and $a_1 = d$, then, for $k\ge 1$, recursively define $a_{k + 1}$ to be the remainder when $a_{k - 1}$ is divided by $a_k$, unless $a_k = 0$, in which case the algorithm terminates. As long as $a_{k+1}$ is defined, let $q_k$ be the unique positive integer which satisfies
\[a_{k - 1} = q_ka_k + a_{k + 1}.\]
It is well known that the Euclidean algorithm always terminates with some $a_\ell = 0$.  We now introduce some notation.
\begin{itemize}
    \item Let $\textbf{a}(n, d)$ denote the \textit{Euclidean sequence} $a_0, \ldots, a_{\ell - 1}$ of nonzero values generated by the Euclidean algorithm applied to $n$ and $d$.
    \item Let $\textbf{q}(n, d)$ denote the \textit{quotient sequence} $q_1, \ldots, q_{\ell - 1}$ generated by the Euclidean algorithm applied to $n$ and $d$. Note that $\textbf{q}(n, d)$ can be computed directly from $\textbf{a}(n, d)$ via $q_i=\left\lfloor\frac{a_{i-1}}{a_i}\right\rfloor$.
    \item Let $s(n, d)$ denote the sum of the elements of the sequence $\textbf{q}(n, d)$.
    \item For integers $1\le m\le n$, let $r(n, m)$ denote the total number of occurrences of $m$ in the sequences $\textbf{q}(n, d)$ as $d$ ranges over the elements of $\{1, \ldots, n - 1\}$ that are relatively prime to $n$.
\end{itemize}

Next, we define the function $X: \mathbb{N}^\ast\rightarrow\mathbb{N}$ by the following recursive rules:
\begin{align*}
X() &= 1,\\
X(q) &= q,\\
X(q_1, \ldots, q_t) &= q_1X(q_2, \ldots, q_t) + X(q_3, \ldots, q_t).
\end{align*}
We can easily see that, for any integers $q_1, \ldots, q_t$ all at least $1$, the following statements are equivalent: $\textbf{q}(n, d)=(q_1, \ldots, q_t)$ and $\textbf{a}(n, d)$ ends in $a_t=1$; and $n = X(q_1, \ldots, q_t)$ and $d = X(q_2, \ldots, q_t)$.  The pair $(n, d)$ can be used to reconstruct $(q_1, \ldots, q_t)$, so the map $\chi: \mathbb{N}^\ast\rightarrow\mathbb{N}^2$ defined via
\[\chi(q_1, \ldots, q_t) = (X(q_1, \ldots, q_t), X(q_2, \ldots, q_t))\]
is injective. It is well-known (see, e.g., the continued fractions interpretation in \cite[Chapter~4, Identity~109]{benjamin2003proofs}) that $X(q_1, \ldots, q_t) = X(q_t, \ldots, q_1)$. Thus, the map $\chi': \mathbb{N}^\ast\rightarrow\mathbb{N}^2$ defined by
\[\chi'(q_1, \ldots, q_t) = (X(q_1, \ldots, q_t), X(q_1, \ldots, q_{t - 1}))\]
is also injective. Note that the images of both $\chi$ and $\chi'$ consist of pairs of relatively prime positive integers.  Also, if $(y,x)$ is in the image of $\chi'$, then $y \geq x$.

\subsection{Proofs}\label{sec:main}
Before we prove Theorem~\ref{thm:number-theoretic} (on sums of quotients), we require a bound on the number of occurrences of large integers in the Euclidean sequences that start with $n$.
\begin{lemma}\label{lem:M-bound}
Fix any constant $K$. Then there is a constant $C = C(K)$ such that if $n\ge 2$ and $1\le M\le K(\log n)^2$, then
\[\sum_{m = M}^n r(n, m)\le\frac{Cn\log n}{M}.\]
\end{lemma}
\begin{proof}
Fix $n$ and $M$. Note that
\[\sum_{m = M}^n r(n, m)\]
counts the triples of positive integers $(a, b, d)$ where $\gcd(n, d) = 1$, $b$ and $a$ appear consecutively (in that order) in the Euclidean sequence $\textbf{a}(n, d)$, and $\frac{b}{a}\ge M$. Let $\mathcal{T}$ denote the set of these triples.  We will count the elements of $\mathcal{T}$ by looking at each possible pair $(a,b)$ separately and counting the values of $d$ such that $(a,b,d) \in \mathcal{T}$.  Note that any such pair $(a,b)$ satisfies $1 \leq a,b \leq n$ and $\frac{b}{a} \geq M$.  Moreover, $\gcd(a,b)=1$ since $\gcd(n,d)=1$.
\\

Fix a pair $(a, b)$ as described in the previous paragraph.  The choices for $d$ such that $(a, b, d)\in\mathcal{T}$ correspond exactly to the Euclidean sequences beginning with $n$ that contain $b$ and $a$ in consecutive positions. Write $b=a_t$ and $a=a_{t+1}$, for some $t \geq 0$.  Note that such a Euclidean sequence is completely determined past $b, a$ because it is simply the Euclidean sequence $\textbf{a}(b, a)$. Thus, each choice for $d$ corresponds to a choice for the Euclidean sequence between $n$ and $b$, which is in turn equivalent to choosing a sequence of quotients $q_1, \ldots, q_t$ within this range. Indeed, given such a sequence $q_1, \ldots, q_t$ with $t>0$, we can reconstruct the earlier terms of the Euclidean sequence by working backwards from $a=a_{t+1},b=a_t$ as follows:
\begin{align*}
    a_{t-1}&=q_tb+a\\
    a_{t-2}&=q_{t-1}(q_tb+a)+b\\
    \vdots \quad & \quad \quad \quad \vdots\\
    (d=)a_1&=q_2(\cdots)+(\cdots)\\
    (n=)a_0&=q_1(\cdots)+(\cdots).
\end{align*}
More explicitly,
$$a_i=X(q_{i+1},\ldots, q_t)b+X(q_{i+1}, \ldots, q_{t-1})a,$$
and, in particular, setting $i=0$ gives
\[n = X(q_1, \ldots, q_t)b + X(q_1, \ldots, q_{t - 1})a.\]
Note that $t = 0$ corresponds to the special case $n = b$ and $d = a$.
\\

Recall from Section~\ref{sec:basic-number-theory} that the map
\[\chi': (q_1, \ldots, q_t)\mapsto (X(q_1, \ldots, q_t), X(q_1, \ldots, q_{t - 1}))\]
is injective and that its range consists of pairs of relatively prime positive integers $(y,x)$, where $y \geq x$.  Thus, excluding the $t=0$ case, the number of choices for $d$ is at most than the number of ordered pairs $(y,x)\in\mathbb{N}^2$ such that
\[n = yb + xa,\]
subject to the additional constraints $\gcd(x, y) = 1$ and $y\ge x$.
\\

Taken together, all of the $t=0$ cases (which require $b=n$ and hence also $a \leq \frac{n}{M}$) contribute at most $\frac{n}{M}$ triples to $\mathcal{T}$.  So, putting everything together, we see that the quantity $\sum_{m = M}^n r(n, m)$ is at most $\frac{n}{M}$ greater than the size of the set
\[\{(a, b, x, y): 1\le a, b, x, y\le n; \frac{b}{a}\ge M; y\ge x; \gcd(a, b) = \gcd(x, y) = 1; n = yb + xa\}.\]
Note that the last condition is equivalent to
\begin{equation}\label{eq:linear}
(y - x)b + x(a + b) = n.
\end{equation}
We now bound the number of solutions to \eqref{eq:linear} (subject, of course, to the other conditions) by conditioning on the size of $b$.
\\

First, consider $b\le\sqrt{n\log n}$. We bound the number of solutions $(y,x)$ to \eqref{eq:linear} for each fixed pair $(a,b)$ with $a\le\frac{b}{M}$, then we sum over these values as $a$ and $b$ vary.  Fix a solution $(x_0,y_0)$ if one exists, and let $(x,y)$ be any solution.  Then $$(y_0-x_0)b+x_0(a+b)=n=(y-x)b+x(a+b),$$
so $$(x_0-x)(a+b)=(y-x-y_0+x_0)b.$$
Since $\gcd(b,a+b)=\gcd(b,a)=1$, we conclude that $x_0-x$ is divisible by $b$, i.e., $$x=x_0+bt$$ for some integer $t$.  Plugging in $x_0-x=-bt$ gives $$y-x=(y_0-x_0)-(a+b)t.$$  Since $0 \leq y-x <\frac{n}{b}$, there are at most $1+\frac{n}{b(a+b)}$ possible values for $t$.  Summing over $a$ and using $\frac{n}{b(a+b)}<\frac{n}{b^2}$, we get that there are at most $$\frac{b}{M}+\frac{n}{bM}$$ triples $(a,x,y)$ for each value of $b$.  Then summing over $b\leq \sqrt{n\log n}$ shows that there are at most
\[\frac{n\log n}{2M} + O\left(\frac{n\log\log n}{M}\right)\]
possible quadruples $(a,b,x,y)$ for $b$ in the desired range.
\\

Second, consider $b > \sqrt{n\log n}$. This time, we choose the pair $(y,x)$ first and then bound the number of possible pairs $(a, b)$. Fix any pair of relatively prime positive integers $(y, x)$ satisfying the above inequalities. We now count the possible pairs $(a, b)$.  The argument from the previous case shows that if $(a', b')$ is a solution to $yb+xa=n$, then every solution is expressible as $(a,b)=(a'+ys,b'-xs)$ for some integer $s$.  The inequality $\frac{b}{a}\ge M$ gives
\[0<(My + x)a\le yb + xa = n,\]
which implies that there are at most $$1+\frac{n}{y(My + x)}<1+\frac{n}{My^2}$$ possible values of $s$, i.e., solutions $(a,b)$.  We now sum over possible pairs $(y,x)$.  Note that \eqref{eq:linear}, in light of the bound on $b$, immediately gives both $y - x < \sqrt{\frac{n}{\log n}}$ and $x < \sqrt{\frac{n}{\log n}}$.  We compute that the number of solutions $(a,b,x,y)$ is bounded above by
\begin{align*}
    \sum_{x=1}^{\sqrt{\frac{n}{\log n}}}\sum_{y=x}^{x+\sqrt{\frac{n}{\log n}}}\left(1+\frac{n}{My^2} \right) &\leq \frac{n}{\log n}+ \frac{n}{M}\sum_{y=1}^{2\sqrt{\frac{n}{\log n}}}\sum_{x=1}^{y}\frac{1}{y^2}\\
     &\leq \frac{n}{\log n}+ \frac{n}{M}\sum_{y=1}^{2\sqrt{\frac{n}{\log n}}}\frac{1}{y}\\
     &=\frac{n}{\log n}+ \frac{n\log n}{2M}+O\left(\frac{n \log\log n}{M}\right).
\end{align*}
Combining the upper bounds from the above two cases for $b$ (and absorbing the $t=0$ case into the error term) gives a total of
\[\frac{n}{\log n}+ \frac{n\log n}{M} + O\left(\frac{n\log\log n}{M}\right).\]
Since $M\le K(\log n)^2$, this bound yields the desired result.
\end{proof}
We are ready to prove the main result of this section. The following is a restatement of Theorem~\ref{thm:number-theoretic} in the language of the previous subsection.

\begin{theorem}
There is an absolute constant $c$ such that the following holds: for every $n\ge 2$, there exists some $1\le d\le n - 1$ relatively prime to $n$ such that $s(n, d)\le c\frac{n}{\phi(n)}\log n\log\log n$.
\end{theorem}
\begin{proof}
Fix some $K$ and $C(K)$ as in Lemma~\ref{lem:M-bound}. Note that $10C\frac{n}{\phi(n)}\log n < K(\log n)^2$ for sufficiently large $n$; this follows from, for instance, the well-known fact \cite{hardy1979introduction} that
\begin{equation}\label{eq:euler-mascheroni}
    \liminf_{n\rightarrow\infty}\frac{\phi(n)}{\frac{n}{\log\log n}} = e^{-\gamma},
\end{equation}
where $\gamma$ is the Euler-Mascheroni constant. We now restrict our attention to this case of sufficiently large $n$, since the constant in the statement of the theorem can be adjusted to accommodate small $n$.
\\

Setting $M=10C\frac{n}{\phi(n)}\log n$ in Lemma~\ref{lem:M-bound} gives
\[\sum_{m = 10C\frac{n}{\phi(n)}\log n}^n r(n, m)\le\frac{\phi(n)}{10}.\]
Let $\mathcal{D}$ be the set of all $d\in\{1, \ldots, n - 1\}$ such that $\gcd(n, d) = 1$ and the elements of $\textbf{q}(n, d)$ are all at most $10C\frac{n}{\phi(n)}\log n$. Then we see that $|\mathcal{D}|\ge 0.9\phi(n)$.
\\

We can now compute:
\begin{align*}
\sum_{d\in\mathcal{D}} s(n, d)&\le\sum_{m = 1}^{10C\frac{n}{\phi(n)}\log n} mr(n, m)\\
&=\sum_{t = 1}^{10C\frac{n}{\phi(n)}\log n}\sum_{m = t}^{10C\frac{n}{\phi(n)}\log n} r(n, m)\\
&\le\sum_{t = 1}^{10C\frac{n}{\phi(n)}\log n}\frac{Cn\log n}{t}\\
&= Cn\log n\log\log n + O\left(n\log n\log\left(10C\frac{n}{\phi(n)}\right)
\right)\\
&= Cn\log n\log\log n + O\left(n\log n\log \log \log n \right),
\end{align*}
where we used Lemma~\ref{lem:M-bound} in the third line and \eqref{eq:euler-mascheroni} in the fifth line. Since $|\mathcal{D}|\ge 0.9\phi(n)$, we also have
\begin{align*}\min_{d\in\mathcal{D}} s(n, d)&\le\frac{C}{0.9}\frac{n}{\phi(n)}\log n\log\log n + O\left(\frac{n}{\phi(n)}\log n\log\log\log n\right),
\end{align*}
and the result immediately follows.
\end{proof}

\section{Large linear extension numbers}\label{sec:large}

The previous three sections have been devoted to showing that $\LE(n)$ contains many small elements.  In this section, we show that $\LE(n)$ does not contain very many large elements.  We begin with a straightforward lemma about posets whose Hasse diagrams (viewed as graphs) are connected.

\begin{lemma}\label{lem:connected}
Let $P$ be an $n$-element poset that cannot be expressed as the disjoint sum of nonempty posets.  Then $\ext(P)\leq (n-1)!$.
\end{lemma}

\begin{proof}
Note that the Hasse diagram of $P$, viewed as a graph, is connected.  Choose a spanning tree of this graph, and delete all other edges; the remaining graph is the Hasse diagram of some poset $P'$.  Clearly, $\ext(P)\leq \ext(P')$, so if suffices to show that $\ext(P')\leq (n-1)!$.
\\

We now proceed by induction on $n$, where the base case $n=1$ is trivial.  Choose an element $x \in P'$ that is a leaf of its Hasse diagram, so that the Hasse diagram of $P' \setminus \{x\}$ is connected.  By hypothesis, $\ext(P' \setminus \{x\}) \leq (n-2)!$.  Let $Q=(P' \setminus \{x\})+\{x\}$ be obtained from $P'$ by ``severing'' the element $x$ from the rest of the poset.  Then $$\ext(Q)=\binom{n}{1} \ext(P' \setminus \{x\}) \ext(\{x\}) \leq n(n-2)!.$$

Without loss of generality, $P'=Q[x<y]$ for some $y \in Q$.  (The case where the added relation is $x>y$ is completely symmetrical.)  We claim that in a uniformly chosen random linear extension of $Q$, the probability that $x<y$ is at most $\frac{n-1}{n}$.  To see this, consider all linear extensions with a fixed relative ordering of the elements other than $x$.  Even if $y$ is the largest of these $n-1$ elements, there is always at least $1$ way to insert $x$ (out of $n$ total ways) so that $x>y$.  Thus, $$\ext(P')\leq \left( \frac{n-1}{n} \right) \ext(Q) \leq (n-1)!,$$ as desired.
\end{proof}

We remark that equality is achieved for the posets $\{x_1\} \oplus (\{x_2\}+\cdots+\{x_n\})$ and $(\{x_2\}+\cdots+\{x_n\}) \oplus \{x_1\}$, the direct sum of a singleton and an antichain of length $n-1$.  The reader can easily verify from the proof of the lemma that these are the only equality cases.
\\

We will also make use of the following more general statement, which reduces to Lemma \ref{lem:connected} when $m=0$.

\begin{lemma}\label{lem:leupperbound}
Let $P=P_1+P_2+\cdots+P_k$, where no $P_i$ can be expressed as the disjoint sum of two smaller posets, be an $n$-element poset.  Suppose $P_k$ is the largest component and $P_1=P_2=\cdots=P_m=\textbf{1}$ are the only singletons among $P_1, \ldots, P_{k-1}$.  Then $\ext(P)\leq \frac{n!}{n-m}$.
\end{lemma}

\begin{proof}
We will prove that the $(n-m)$-element poset $Q=P_{m+1}+\cdots+P_k$ satisfies $\ext(Q) \leq (n-m-1)!$, then the statement follows from $$\ext(P)=(n-m+1)(n-m+2)\cdots (n) \ext(Q)\leq \frac{n!}{n-m}.$$

We proceed by induction on $k-m$, where the base case $k-m=1$ is Lemma \ref{lem:connected}.  For the induction step, we compute 
\begin{align*}
\ext(Q) &=\binom{n-m}{|P_k|} \ext(P_{m+1}+\cdots+P_{k-1}) \ext(P_k)\\
 &\leq \binom{n-m}{|P_k|} \left( \frac{(n-m-|P_k|)!}{n-m-|P_k|} \right) (|P_k|-1)!\\
 &=\frac{(n-m)!}{|P_k|(n-m-|P_k|)},
\end{align*}
where the second line used the induction hypothesis.
\\

If $|P_k|=1$, then it must be the case that $P$ is antichain and $m=n-1$, so the statement is true by direct computation.  Otherwise, both $|P_k|\geq 2$ and $n-m-|P_k| \geq 2$ (since $Q \setminus P_k$ still has at least one connected component, which by construction must contain at least $2$ elements).  Thus, $|P_k|(n-m-|P_k|) \geq |P_k|+(n-m-|P_k|)=n-m$ establishes the claim and completes the proof.
\end{proof}

We now use these lemmas to obtain restrictions on possible large linear extension numbers.  In particular, we show that for near the top of the range $[1,n!]$,  $\LE(n)$ looks like a dilation of the large values of some suitable $\LE(n-r)$.

\begin{theorem}\label{thm:big}
For integers $r<n$, we have
$$\LE(n) \cap \left(\frac{n!}{r+1}, n!\right]=\left\{\left( \frac{n!}{r!}\right) \ell: \ell \in \LE(r) \cap \left(\frac{r!}{r+1}, r!\right] \right\}$$
as an equality of sets.
\end{theorem}

\begin{proof}
To see the first inclusion, let $P$ be any $n$-element poset satisfying $\ext(P)>\frac{n!}{r+1}$.  Decompose $P=P_1+P_2+\cdots+P_k$ and define $m$ as in Lemma \ref{lem:leupperbound}.  Then $\frac{n!}{n-m} \geq \ext(P) > \frac{n!}{r+1}$ implies that $r+1> n-m$ and  $m \geq n-r$.  Note that $Q=P_{n-r+1}+P_{n-r+2}+\cdots+P_k$ is an $r$-element poset.  Finally, we have $$\ext(P)=(r+1)(r+2)\cdots (n) \ext(Q)=\left( \frac{n!}{r!} \right) \ext(Q),$$ as desired.
\\

To see the other inclusion, let $Q$ be any $r$-element poset.  Then let $P$ be the disjoint sum of $Q$ and an antichain of size $n-r$, so that $\ext(P)=\left( \frac{n!}{r!} \right) \ext(Q)$, as above.
\end{proof}

This theorem tells us that the elements of $\LE(n)$ become sparser as one approaches the upper bound $n!$.  Note that the $r=2$ case recovers the obvious fact that the two largest elements of $\LE(n)$ are $n!$ and $\frac{n!}{2}$.  Similarly, the $r=3$ case tells us that the next largest element is $\frac{n!}{3}$, and so on.  The bound on $\LE(n) \cap ((n-1)!, n!]$ in Theorem~\ref{thm:large}, which we restate here as a corollary, follows immediately.

\begin{corollary}\label{cor:bigbound}
For all $n \geq 8$, we have $|\LE(n) \cap ((n-1)!,n!]|<(n-3)!$.
\end{corollary}

\begin{proof}
Let $N(n)=|\LE(n) \cap ((n-1)!,n!]|$.  For any $n \geq 2$, the $r=n-1$ case of Theorem \ref{thm:big} gives:
\begin{align*}
N(n) &=\left|\LE(n-1) \cap \left(\frac{(n-1)!}{n},(n-1)!\right]\right|\\
 &=\left|\LE(n-1) \cap \left(\frac{(n-1)!}{n},(n-2)!\right]\right|+|\LE(n-1) \cap ((n-2)!,(n-1)!]|\\
 &\leq \left\lceil (n-2)!-\frac{(n-1)!}{n} \right\rceil +N(n-1)\\
 &=\left\lceil \frac{(n-2)!}{n} \right\rceil +N(n-1).
\end{align*}
This inequality, together with $N(2)=1$, lets us compute  $N(3)\leq 2$, $N(4) \leq 3$, $N(5)\leq 5$, $N(6) \leq 9$, $N(7) \leq 27$, $N(8) \leq 117<5!$.  The corollary then follows (by induction) from the observation that for all $n \geq 9$, we have
\begin{align*}
\left\lceil \frac{(n-2)!}{n} \right\rceil +(n-4)! &< \frac{(n-2)!}{n}+1+(n-4)!\\
 &=(n-3)! \left\lbrack \frac{n-2}{n}+\frac{1}{(n-3)!}+\frac{1}{n-3} \right\rbrack\\
  &<(n-3)! \left\lbrack 1-\frac{2}{n}+\frac{1}{2n} +\frac{3}{2n} \right\rbrack\\
 &<(n-3)!.
\end{align*}
\end{proof}

\section{Conclusion}\label{sec:conclusion}

Theorems~\ref{thm:small} and \ref{thm:big} show a sense in which $\LE(n)$ contains small values ``more'' than it contains very large values.  We have, however, barely scratched the surface.  We present the following list of questions, problems, and conjectures about the structure of $\LE(n)$.

\begin{question}
Let $M(n)$ denote the smallest positive integer that is not in $\LE(n)$. We have shown that $M(n)$ grows as at least $\exp\left(O\left(\frac{n}{\log n}\right)\right)$. How much can this growth rate be improved?  What upper bounds can be established?
\end{question}
Note that any improvement to Theorem~\ref{thm:number-theoretic} immediately translates into a better lower bound on $M(n)$. In particular, we conjecture that Theorem~\ref{thm:number-theoretic} can be improved to what is clearly the best possible.
\begin{conjecture}\label{conj:euclidean}
There is an absolute constant $c$ such that the following holds: for every $n\ge 2$, there exists some $1\le d\le n - 1$ relatively prime to $n$ such that $s(n,d)\leq c\log n$.
\end{conjecture}
Conjecture~\ref{conj:euclidean} implies the lower bound on our conjectured growth rate for $M(n)$.
\begin{conjecture}\label{conj:extensions}
$M(n) = \exp(\Omega(n))$.
\end{conjecture}
Recall that the posets used in the proof of Theorem~\ref{thm:small} are all width-$2$.  Inspired by this fact, we conjecture that width-$2$ posets asymptotically achieve $M(n)$.
\begin{conjecture}
Let $M_2(n)$ be the smallest positive integer that is not the linear extension number of a width-$2$ poset. Then $M_2(n) = \exp(\Omega(n))$.
\end{conjecture}
It is easy to show that there are $\exp(O(n))$ width-$2$ posets of size $n$, each with $\exp(O(n))$ linear extensions.  As such, this conjecture is the best possible for width-$2$ posets (or, in fact,  any constant-width posets).
\\

We conclude with two more general problems for future research.
\begin{problem}
\textup{Estimate $|\LE(n)|$ and other statistics.}
\end{problem}
\begin{problem}
\textup{For fixed $c$, estimate the largest $R_c(n)$ such that $|\LE(n) \cap [1,R]| \geq cR$. In particular, is $R_c(n)$ quite different from $M(n)$ in terms of its asymptotic growth rate?}
\end{problem}

\bibliographystyle{plain}
\bibliography{main}

\end{document}